\documentclass{amsart}
\usepackage{fullpage}
\usepackage{amsmath,amssymb, amscd}

\newtheorem{theorem}{Theorem}[section]
\newtheorem{lemma}[theorem]{Lemma}
\newtheorem{definition}[theorem]{Definition}
\newtheorem{proposition}[theorem]{Proposition}
\newtheorem{corollary}[theorem]{Corollary}
\newtheorem{conjecture}[theorem]{Conjecture}
\theoremstyle{remark}
\newtheorem{remark}[theorem]{Remark}
\newtheorem{question}[theorem]{Question}
\newtheorem{example}[theorem]{Example}

\newcommand {\bfe} {\mathbf{e}}
\newcommand {\bfn} {\mathbf{n}}
\newcommand {\bfv} {\mathbf{v}}
\newcommand {\bflim} {\mathbf{lim}}

\newcommand {\bbC} {\mathbb{C}}
\newcommand {\bbF} {\mathbb{F}}
\newcommand {\bbN} {\mathbb{N}}
\newcommand {\bbQ} {\mathbb{Q}}
\newcommand {\bbR} {\mathbb{R}}
\newcommand {\bbZ} {\mathbb{Z}}

\newcommand {\rmN} {\mathrm{N}}

\newcommand {\calD} {\mathcal{D}}
\newcommand {\calE} {\mathcal{E}}
\newcommand {\calL} {\mathcal{L}}
\newcommand {\calM} {\mathcal{M}}

\newcommand {\Pic}{\mathop{\mathrm{Pic}}}
\newcommand {\Gal}{\mathop{\mathrm{Gal}}}
\newcommand {\chr}{\mathop{\mathrm{char}}}
\newcommand {\Res}{\mathop{\mathrm{Res}}}

\newcommand {\sign}{\mathop{\mathrm{sign}}}

\renewcommand {\Re}{\mathop{\mathrm{Re}}}

\title{Asymptotic properties of zeta functions over finite fields}
\author{Alexey Zykin}
\address{
Alexey Zykin
\newline \indent
National Research University Higher School of Economics\newline \indent
AG Laboratory NRU HSE \newline \indent
117312, Vavilova st., 7, Moscow, Russia
\newline \indent
Laboratoire Poncelet (UMI 2615)
\newline \indent
Institute for Information Transmission Problems of the Russian Academy of Sciences
}
\email{alzykin@gmail.com}
\date{}
\thanks{The author was partially supported by AG Laboratory NRU HSE, RF government grant, ag.  11.G34.31.0023, by the grants RFBR 12-01-31280 mola, RFBR 11-01-00393-a, RFBR 12-01-92697-Inda, by Simons--IUM fellowship, and by Dmitry Zimin's ``Dynasty'' foundation.}

\begin{document}

\begin{abstract}
In this paper we study asymptotic properties of families of zeta and $L$-functions over finite fields. We do it in the context of three main problems: the basic inequality, the Brauer--Siegel type results and the results on distribution of zeroes. We generalize to this abstract setting the results of Tsfasman, Vl\u adu\c t and Lachaud, who studied similar problems for curves and (in some cases) for varieties over finite fields. In the classical case of zeta functions of curves we extend a result of Ihara on the limit behaviour of the Euler--Kronecker constant. Our results also apply to $L$-functions of elliptic surfaces over finite fields, where we approach the Brauer--Siegel type conjectures recently made by Kunyavskii, Tsfasman and Hindry.
\end{abstract}

\maketitle
\section{Introduction}
The origin of the asymptotic theory of global fields and their zeta-functions can be traced back to the following classical question: what is the maximal number of points $N_q(g)$ on a smooth projective curve of genus $g$ over the finite field $\bbF_q$. The question turns out to be difficult and a wide variety of methods has been used for finding both upper and lower bounds.

The classical bound of Weil stating that
$$|N_q(g)-q-1|\leq 2g\sqrt{q}$$
though strong turns out to be far from optimal. A significant improvement for it when $g$ is large was obtained by Drinfeld and Vl\u adu\c t (\cite{DV}). Namely, they proved that $\limsup\limits_{g\to\infty} \frac{N_q(g)}{g} \leq \sqrt{q}-1.$

This inequality was a starting point for an in-depth study of asymptotic properties of curves over finite fields and of their zeta functions initiated by Tsfasman and Vl\u adu\c t. This work went far beyond this initial inequality and has led to the introduction of the concept of limit zeta function which turned out to be very useful \cite{TV97}. It also had numerous applications to coding theory via the so-called algebraic geometric codes (see, for example, the book \cite{TVN} for some of them).

The above mentioned study of limit zeta functions for families curves involves three main topics:
\begin{enumerate}
\item The basic inequality, which is a generalization of the Drinfeld -- Vl\u adu\c t inequality on the number of points on curves.
\item Brauer--Siegel type results which are the extensions of the classical Brauer--Siegel theorem describing the asymptotic behaviour of the class numbers and of the regulators in families of number fields. Here asymptotic properties of the special values of zeta functions of curves (such as the order of the Picard group) are studied.
\item The distribution of zeroes of zeta functions (Frobenius eigenvalues) in families of curves.
\end{enumerate}

There are at least two main directions in the further study of these topics. First, one may ask what are the number field counterparts of these results (for number fields and function fields are regarded by many as facets of a single gemstone). The translation of these results to the number field case is the subject of the paper \cite{TV02}. The techniques turns out to be very analytically involved but the reward is no doubts significant as the authors managed to resolve some of the long standing problems (such as the generalization of the Brauer--Siegel theorem to the asymptotically good case, that is when the ratio $n_K/\log |D_K|$ of the degree to the logarithm of the discriminant does not tend to zero) as well as to improve several difficult results (Odlyzko--Serre inequalities for the discriminant, Zimmert's bound for regulators).

Second, one may ask what happens with higher dimensional varieties over finite fields. Here the answers are less complete. The first topic (main inequalities) was extensively studied in \cite{LT}. The results obtained there are fairly complete, though they do not directly apply to $L$-functions (such as $L$-functions of elliptic curves over function fields). The second topic is considerably less developed though it received some attention in the recent years in the case of elliptic surfaces \cite{HP}, \cite{KT} and in the case of zeta functions of varieties over finite fields \cite{Z1}. As for the results on the third topic one can cite a paper by Michel \cite{Mi} where the case of elliptic curves over $\bbF_q(t)$ is treated. Quite a considerable attention was devoted to some finer questions related to the distribution of zeroes \cite{KS}. However, to our knowledge, not a single result of this type for asymptotically good families of varieties is known.

The goal of our paper is to study the above three topic in more generality separating fine arithmetic considerations from a rather simple (in the function field case) analytic part. We take the axiomatic approach, defining a class of $L$-functions to which our results may be applicable. This can be regarded as the function field analogue of working with the Selberg class in characteristic zero, though obviously the analytic contents in the function case is much less substantial (and often times even negligible). In our investigations we devote more attention to the second and the third topics as being less developed then the first one. So, while giving results on the basic inequality, we do not seek to prove them in utmost generality (like in the paper \cite{LT}). We hope that this allows us to gain in clarity of the presentation as well as to save a considerable amount of space.

We use families of $L$-functions of elliptic curves over function fields as our motivating example. After each general statement concerning any of the three topics we specify what concrete results we get for zeta functions of curves, zeta functions of varieties over finite fields, and $L$-functions of elliptic curves over function fields. In the study of the second topic we actually manage to find something new even in the classical case of zeta functions of curves, namely we prove a statement on the limit behaviour of zeta functions of which the Brauer--Siegel theorem from \cite{TV97} is a particular case (see theorem \ref{BSGen} and corollary \ref{BSCurves}). We also reprove and extend some of Ihara's results on  Euler--Kronecker constant of function fields \cite{Ih} incorporating them in the same general framework of limit zeta functions (see corollary \ref{EulerKronecker}). Our statements about the distribution of zeroes (theorem \ref{zerodistrib} and corollary \ref{zeroesMi}) imply in the case of elliptic curves over function fields a generalization of a result due to Michel \cite{Mi} (however, unlike us, Michel provides a rather difficult estimate for the error term).

Here is the plan of our paper. In section \ref{sectzeta} we present the axiomatic framework for zeta and $L$-functions with which we will be working, then we give the so called explicit formulae for them. In the end of the section we introduce several particular examples coming from algebraic geometry (zeta functions of curves, zeta functions of varieties over finite fields, $L$-functions of elliptic curves over function fields) to which we will apply the general results. Each further section contains a subsection where we show what the results on abstract zeta and $L$-functions give in these concrete cases. In section \ref{sectasympt} we outline the asymptotic approach to the study of zeta and $L$-functions, introducing the notions of asymptotically exact and asymptotically very exact families. Section \ref{sectineq} is devoted to the proof of several versions of the basic inequality. The study of the Brauer--Siegel type results is undertaken in section \ref{sectBS}. In the same section we show how these results imply the formulae for the asymptotic behaviour of the invariants of function fields generalizing the Euler--Kronecker constant (corollary \ref{EulerKronecker}) and a certain bound towards the conjectures of Kunyavskii--Tsfasman and Hindry--Pacheko (theorem \ref{KunTsfHin}). We prove the zero distribution results in section \ref{sectzero}. There we also give some applications to the distribution of zeroes and the growth of analytic ranks in families of elliptic surfaces (corollaries \ref{zeroesMi} and \ref{rankboundell}). Finally, in section \ref{sectquestions} we discuss some possible further development as well as open questions.

\section{Zeta and $L$-functions}
\label{sectzeta}
\subsection{Definitions}
Let us define the class $L$-functions we will be working with. Let $\bbF_q$ be the finite field with $q$ elements.
\begin{definition}
An $L$-function $L(s)$ over a finite field $\bbF_q$ is a holomorphic function in $s$ such that for $u=q^{-s}$ the function $\calL(u)=L(s)$ is a polynomial with real coefficients, $\calL(0)=1$ and all the roots of $\calL(u)$ are on the circle of radius $q^{-\frac{w}{2}}$ for some non-negative integer number $w.$
\end{definition}

We will refer to the last condition in the definition as the Riemann hypothesis for $L(s)$ since it is the finite field analogue of the classical Riemann hypothesis for the Riemann zeta function. The number $w$ in the definition of an $L$-function will be called its {\it weight}. We will also say that the degree $d$ of the polynomial $\calL(u)$ is the {\it degree} of the $L$-function $L(s)$ (it should not be confused with the degree of an $L$-function in the analytic number theory, where it is taken to be the degree of the polynomial in its Euler product).

The logarithm of an $L$-function has a Dirichlet series expansion
$$\log L(s)=\sum_{f=1}^{\infty} \frac{\Lambda_f}{f} q^{-f s},$$
which converges for $\Re s > \frac{w}{2}.$ For the opposite of the logarithmic derivative we get the formula:
$$-\frac{L'(s)}{L(s)}=\sum_{f=1}^{\infty} (\Lambda_f \log q) \, q^{-f s}= u \,\frac{\calL'(u)}{\calL(u)} \log q.$$

There is a functional equation for $L(s)$ of the form
\begin{equation}
\label{funceq}
L(w-s)=\omega q^{(\frac{w}{2}-s)d}L(s),
\end{equation}
where $d = \deg \calL(u)$ and $\omega =\pm 1$ is the root number. This can be proven directly as follows. Let $\calL(u)=\prod\limits_{i=1}^{d}\left(1-\frac{u}{\rho_i}\right).$ Then
\begin{equation*}
\calL\left(\frac{1}{u q^w}\right)=\prod_{\rho}\left(1-\frac{1}{\rho u q^w}\right)=\prod_{\rho}\rho \cdot q^{w d} u^d \prod_{\rho}\left(\frac{u}{\bar{\rho}}-1\right)=
(-1)^{d-t} q^{\frac{w d}{2}} u^d \prod_{\rho}\left(1-\frac{u}{\rho}\right),
\end{equation*}
where $t$ is the multiplicity of the root $-q^{w/2}$. We used the fact that all coefficients of $\calL(u)$ are real, so its non-real roots come in pairs $\rho$ and $\bar{\rho},$ $\rho\bar{\rho}=q^w.$

\begin{definition}
A zeta function $\zeta(s)$ over a finite field $\bbF_q$ is a product of $L$-functions in powers $\pm 1:$
$$\zeta(s)=\prod_{i=0}^{w}L_i(s)^{\epsilon_i},$$
where $\epsilon_i\in\{-1, 1\},$ $L_i(s)$ is an $L$-function of weight $i.$
\end{definition}

For the logarithm of a zeta function we also have the Dirichlet series expansion:
$$\log \zeta(s)=\sum\limits_{f=1}^{\infty} \frac{\Lambda_f}{f} q^{-f s}$$
which is convergent for $\Re s > \frac{w}{2}.$

\subsection{Explicit formulae}
In this subsection we will derive the analogues of Weil and Stark explicit formulae for our zeta and $L$-functions. The proofs of the Weil explicit formula can be found in \cite{Ser} for curves and in \cite{LT} for varieties over finite fields. An explicit formula for $L$-functions of elliptic surfaces is proven in \cite{Bru}. In our proof we will follow the latter exposition.

Recall that our main object of study is $\zeta(s)=\prod\limits_{i=0}^{w}L_i(s)^{\epsilon_i}$ a zeta function with $L_i(s)$ given by $$L_i(s)=\prod\limits_{j=1}^{d_i}\left(1-\frac{q^{-s}}{\rho_{i j}}\right).$$
As before, we define $\Lambda_f$ via the relation $\log \zeta(s)=\sum\limits_{f=1}^{\infty} \frac{\Lambda_f}{f} q^{-f s}.$

\begin{proposition}
\label{explicit}
Let $\bfv=(v_f)_{f\geq 1}$ be a sequence of real numbers and let $\psi_{\bfv}(t)=\sum\limits_{f=1}^{\infty} v_f t^f$. Let $\rho_{\bfv}$ be the radius of convergence of the series $\psi_{\bfv}(t).$ Assume that $|t| < q^{-w/2}\rho_{\bf v},$ then
$$\sum_{f=1}^{\infty}\Lambda_f v_f t^f=-\sum_{i=0}^w \epsilon_i \sum_{j=1}^{d_i} \psi_{\bfv}(q^{i}\rho_{i j}t).$$
\end{proposition}

\begin{proof}
Let us prove this formula for $L$-functions. The formula for zeta functions will follow by additivity.

The simplest is to work with $\calL(u)=\prod\limits_{j=1}^{d}\left(1-\frac{u}{\rho}\right).$ The coefficient of $u^f$ in $-u\calL'(u)/\calL(u)$ is seen to be $\sum\limits_{\rho}\rho^{-f}$ for $f\geq 1.$ From this we derive the equality:
$$\sum\limits_{\rho}\rho^{-f} = -\Lambda_f.$$
The map $\rho \mapsto (q^{w}\rho)^{-1}$ permutes the zeroes $\{\rho\},$ thus for any $f\geq 1$ we have:
$$\sum_{\rho}(q^{w}\rho)^f=-\Lambda_{f}.$$
Multiplying the last identity by $v_f t^f$ and summing for $f=1, 2, \dots$ we get the statement of the theorem.
\end{proof}

From this theorem one can easily get a more familiar version of the explicit formula (like the one from \cite{Ser} in the case of curves over finite fields).

\begin{corollary}
\label{corexplicit}
Let $L(s)$ be an $L$-function, with zeroes $\rho=q^{-w/2}e^{i\theta},$ $\theta \in [-\pi,\pi].$ Let $f\colon [-\pi,\pi]\to \bbC$ be an even trigonometric polynomial
$$f(\theta)=v_0+2\sum_{n=1}^{Y}v_n \cos(n\theta).$$
Then we have the explicit formula:
\begin{equation*}
\sum_\theta f(\theta) = v_0 d - 2 \sum_{f=1}^{Y}v_f\Lambda_f q^{-\frac{w f}{2}}.
\end{equation*}
\end{corollary}
\begin{proof}
We put $t=q^{-\frac{w}{2}}$ in the above explicit formula and notice that the sum over zeroes can be written using $\cos$ since all the non-real zeroes come in complex conjugate pairs.
\end{proof}

In the next sections we will also make use of the so called Stark formula (which borrows its name from its number field counterpart from \cite{Sta}).

\begin{proposition}
\label{Stark}
For a zeta function $\zeta(s)$ we have:
\begin{equation*}
\frac{1}{\log q}\frac{\zeta'(s)}{\zeta(s)}= \sum_{i=0}^{w}\epsilon_i\sum_{j=1}^{d_i}\frac{1}{q^{s}{\rho_{ij}}-1}=-\frac{1}{2}\sum_{i=0}^{w}\epsilon_i d_i +\frac{1}{\log q} \sum_{i=0}^{w}\epsilon_i\sum_{L_i(\theta_{i j})=0}\frac{1}{s-\theta_{i j}},
\end{equation*}
we assume that $q^{-\theta_{ij}}=\rho_{ij},$ the sum is taken over all possible roots $\theta_{i j}$ counted with multiplicity.
\end{proposition}
\begin{proof}
The first equality is a trivial consequence of the formulae expressing $\calL_i(u)$ as polynomials in $u.$

The second equality follows from the following series expansion:
\begin{equation*}
\frac{\log q}{\rho^{-1}q^s -1}+\frac{\log q}{2}=\lim_{T\to\infty}\sum_{\substack{q^{\theta}=\rho \\ |\theta|\leq T}} \frac{1}{s-\theta}.
\end{equation*}
\end{proof}

\subsection{Examples}
We have in mind three main types of examples: zeta functions of curves over finite fields, zeta functions of varieties over finite fields and $L$-functions of elliptic curves over function fields.

\begin{example} [Curves over finite fields]
Let $X$ be an absolutely irreducible smooth projective curve of genus $g$ over the finite field $\bbF_q$ with $q$ elements. Let $\Phi_f$ be the number of points of degree $f$ on $X.$ The zeta function of $X$ is defined for $\Re s > 1$ as
$$\zeta_X(s)=\prod_{f=1}^{\infty}(1-q^{-f s})^{-\Phi_f}.$$
It is known that $\zeta_X(s)$ is a rational function in $u=q^{-s}.$ Moreover,
\begin{equation*}
\zeta_{K}(s)=\frac{\prod\limits_{j=1}^{g}\left(1-\frac{u}{\rho_j}\right)\left(1-\frac{u}{\bar{\rho}_j}\right)}{(1-u)(1-qu)},
\end{equation*}
and $|\rho_j|=q^{-\frac{1}{2}}$ (Weil's theorems).
It can easily be seen that in this case $\Lambda_f=N_f(X)$ is the number of points on $X\otimes_{\bbF_q}\bbF_{q^f}$ over $\bbF_{q^f}.$ A very important feature of this example which will be lacking in general is that $\Lambda_f \geq 0$ for all $f.$

Though $\zeta_X(s)$ is not an $L$-function, in all asymptotic considerations the denominator will be irrelevant and it will behave as an $L$-function.

This example will serve as a motivation in most of our subsequent considerations, for most (but not all, see section \ref{sectBS}) of the results we derive for general zeta and $L$-functions are known in this setting.
\end{example}

\begin{example} [Varieties over finite fields]
Let $X$ be a non-singular absolutely irreducible projective variety of dimension $n$ defined over a finite field $\bbF_q.$ Denote by $|X|$ the set of closed points of $X$. We put $X_f=X\otimes_{\bbF_q}\bbF_{q^f}$ and $\bar{X}=X\otimes_{\bbF_q} \overline{\bbF}_q$. Let $\Phi_f$ be the number of points of $X$ having degree $f$, that is $\Phi_{f} = |\{v\in |X| \mid \deg(v)=f\}|$. The number $N_f$ of $\bbF_{q^f}$-points of the variety $X_f$ is equal to $N_f=\sum\limits_{m\mid f}m\Phi_{m}.$

Let $b_s(X)=\dim_{\bbQ_l} H^s(\bar{X}, \bbQ_l)$ be the $l$-adic Betti numbers of $X.$ The zeta function of $X$ is defined for $\Re(s)>n$ by the following Euler product:
$$\zeta_X(s)=\prod_{v \in |X|}\frac {1}{1-(\rmN v)^{-s}}=\prod_{f=1}^{\infty}(1-q^{-f s})^{-\Phi_{f}},$$
where $\rmN v=q^{-\deg v}$. If we set $Z_X(u)=\zeta_X (s)$ with $u=q^{-s}$ then the function $Z_X(u)$ is a rational function of $u$ and can be expressed as
$$Z_X(u)=\prod_{i=0}^{2n} P_i(X, u)^{(-1)^{i-1}},$$
where
$$P_i(X, u)=\prod_{j=1}^{d_i}\left(1-\frac{u}{\rho_{i j}}\right),$$
and $|\rho_{ij}|=q^{-i/2}$ (Weil's conjectures proven by Deligne). Moreover, $P_0(X, u)=1-u$ and $P_{2n}(X, u)=1-q^n u$. As before, we have that $\Lambda_f=N_f(X)\geq 0.$

The previous example is obviously included in this one. However, it is better to separate them as in the case of zeta functions of general varieties over finite fields much less is known. One more reason to distinguish between these two examples is that, whereas zeta functions of curves asymptotically behave as $L$-functions, zeta functions of varieties are "real" zeta functions. Thus there is quite a number of properties that simply do not hold in general (for example, some of those connected to the distribution of zeroes).
\end{example}

\begin{example} [Elliptic curves over function fields]
Let $E$ be a non-constant elliptic curve over a function field $K=\bbF_q(X)$ with finite constant field $\bbF_q.$ The curve $E$ can also be regarded as an elliptic surface over $\bbF_q.$ Let $g$ be the genus of $X.$ Places of $K$ (that is points of $X$) will be denoted by $v.$  Let $d_v=\deg v,$ $|v|=\rmN v=q^{\deg v}$ and let $\bbF_v=\bbF_{\rmN v}$ be the residue field of $v.$

For each place $v$ of $K$ we define $a_v$ from $|E_v(\bbF_v)|=|v|+1-a_v,$ where $|E_v(\bbF_v)|$ is the number of points on the reduction $E_v$ of the curve $E$. The local factors $L_v(s)$ are defined by
$$L_{v}(s)=\begin{cases}
(1-a_{v} |v|^{-s}+|v|^{1-2s})^{-1},&\text{if $E_v$ is non-singular;}\\
(1-a_{v} |v|^{-s})^{-1},&\text{otherwise.}
\end{cases}$$

We define the global $L$-function $L_E(s)=\prod\limits_{v}L_v(s).$ The product converges for $\Re s >\frac{3}{2}$ and defines an analytic function in this half-plane. Define the conductor $N_E$ of $E$ as the divisor $\sum\limits_v n_v v$ with $n_v=1$ at places of multiplicative reduction, $n_v=2$ at places of additive reduction for $\chr \bbF_q>3$ (and possibly larger when $\chr \bbF_q=2$ or $3$) and $n_v=0$ otherwise. Let $n_E=\deg N_E=\sum\limits_v n_v \deg v.$

It is known (see \cite{Bru}) that $L_E(s)$ is a polynomial $\calL_E(u)$ in $u=q^{-s}$ of degree $n_E+4g-4.$ The polynomial $\calL_E(u)$ has real coefficients, satisfies $\calL_E(0)=1$ and all of its roots have absolute value $q^{-1}.$

Let $\alpha_v, \bar{\alpha}_v$ be the roots of the polynomial $1-a_v t +|v|t^2$ for a place $v$ of good reduction and let $\alpha_v=a_v$ and $\bar{\alpha}_v=0$ for a place $v$ of bad reduction. Then from the definition of $L_E(s)$ one can easily deduce that
\begin{equation}
\label{coeffell}
\Lambda_f=\sum_{m d_v=f} d_v(\alpha_v^m + \bar{\alpha}_v^m),
\end{equation}
the sum being taken over all places $v$ of $K$ and $m\geq 1$ such that $m \deg v =f.$

This example will be the principal one in the sense that all our results on $L$-functions are established in the view to apply them to this particular case. These $L$-functions are particularly interesting from the arithmetic point of view, especially due to the connection between the special value of such an $L$-function at $s=1$ and the arithmetic invariants of the elliptic curve (the order of the Shafarevich--Tate group and the regulator) provided by the Birch and Swinnerton-Dyer conjecture.

We could have treated the more general example of abelian varieties over function fields. However, we prefer to restrict ourselves to the case elliptic curves to avoid technical complications.
\end{example}

\section{Families of zeta and $L$-functions}
\label{sectasympt}
\subsection{Definitions and basic properties}
We are interested in studying sequences of zeta and $L$-functions. Let us fix the finite field $\bbF_q.$

\begin{definition}
We will call a sequence $\{L_k(s)\}_{k=1\dots\infty}$ of $L$-functions a family if they all have the same weight $w$ and the degree $d_k$ tends to infinity.
\end{definition}

\begin{definition}
We will call a sequence $\{\zeta_k(s)\}_{k=1\dots\infty}=\left\{\prod\limits_{i=0}^{w}L_{i k}(s)^{\epsilon_i}\right\}_{k=1\dots\infty}$ of zeta functions a family if the total degree $\tilde{d}_k=\sum\limits_{i=0}^{w} d_{i k}$ tends to infinity. Here $d_{i k}$ are the degrees of the individual $L$-functions $L_{i k}(s)$ in $\zeta_k(s).$
\end{definition}

\begin{remark}
In the definition of a family of zeta functions we assume that $w=w_k$ and $\epsilon_i=\epsilon_{i k}$ are the same for all $k.$ Obviously, any family of $L$-functions is at the same time a family of zeta functions.
\end{remark}

\begin{definition}
\label{asexact}
A family $\{\zeta_k(s)\}_{k=1\dots\infty}$ of zeta or $L$-functions is called asymptotically exact if the limits
$$\delta_i=\delta_i(\{\zeta_k(s)\})=\lim_{k\to \infty} \frac{d_{i k}}{\tilde{d}_k} \ \text{ and } \  \lambda_f=\lambda_f(\{\zeta_k(s)\})=\lim_{k\to \infty} \frac{\Lambda_{f k}}{\tilde{d}_k}$$
exist for each $i = 0, \dots, w$ and each $f\in \bbZ,$ $f\geq 1.$ It is called asymptotically bad if $\lambda_f=0$ for any $f$ and asymptotically good otherwise.
\end{definition}

The following (easy) proposition will be important.

\begin{proposition}
\label{neglig}
Let $L(s)$ be an $L$-function. Then
\begin{enumerate}
\item for each $f$ we have the bound $|\Lambda_f| \leq q^{\frac{wf}{2}} d;$
\item there exists a number $C(q, w, s)$ depending on $q,$ $w$ and $s$ but not on $d$ such that $|\log L(s)|\leq C(q, w, s) d$ for any $s$ with $\Re s \neq\frac{w}{2}.$ The number $C(q, w, s)$ can be chosen independent of $s$ if $s$ belongs to a vertical strip $a\leq\Re s\leq b,$ $\frac{w}{2}\notin [a, b].$
\end{enumerate}
\end{proposition}

\begin{proof}
To prove the first part we use proposition \ref{explicit}. Applying it to the sequence consisting of one non-zero term we get:
\begin{equation}
\label{zerosums}
\Lambda_{f}=-\sum_{\calL(\rho)=0}q^{w f}\rho^{f}.
\end{equation}
The absolute value of the right hand side of this equality is bounded by $q^{\frac{w f}{2}} d.$

To prove the second part we assume first that $\Re s=\epsilon + \frac{w}{2}> \frac{w}{2}.$ We have the estimate:
$$|\log L(s)|=\left|\sum_{f=1}^\infty \frac{\Lambda_f}{f}q^{-fs}\right|\leq \sum_{f=1}^\infty \frac{d}{f}\cdot q^{\frac{w f}{2}} \cdot q^{-f\Re s} \leq d \sum_{f=1}^\infty \frac{1}{f q^{\epsilon f}}.$$

For $\Re s < \frac{w}{2}$ we use the functional equation (\ref{funceq}).
\end{proof}

\begin{proposition}
Any family of zeta and $L$-functions contains an asymptotically exact subfamily.
\end{proposition}
\begin{proof}
We note that both $\frac{d_{i k}}{\tilde{d}_k}$ and $\frac{\Lambda_{f k}}{\tilde{d}_k}$ are bounded. For the first expression it is obvious and the second expression is bounded by proposition \ref{neglig}.  Now we can use the diagonal method to choose a subfamily for which all the limits exist.
\end{proof}

As in the case of curves over finite fields we have to single out the factors in zeta functions which are asymptotically negligible. This can be done using proposition \ref{neglig}.

\begin{definition}
Let $\{\zeta_k(s)\}$ be an asymptotically exact family of zeta functions. Define the set $I\subset \{0,\dots, w\}$ by the condition $i\in I$ if and only if $\delta_i=0.$ We define $\zeta_{\bfn, k}(s)= \prod\limits_{i\in I} L_{i k}(s)^{\epsilon_i}$ the negligible part of $\zeta_{k}(s)$ and $\zeta_{\bfe, k}(s)= \prod\limits_{i\in \{0,\dots, w\}\setminus I} L_{i k}(s)^{\epsilon_i}$ the essential part of $\zeta_{k}(s).$  Define also $w_{\bfe}=\max \{i\in \{0,\dots, w\}\setminus I\}.$
\end{definition}

\begin{remark}
The functions $\zeta_{\bfn, k}(s)$ and $\zeta_{\bfe, k}(s)$ make sense only for families of zeta functions and not for individual zetas. We also note that the definitions of the essential and the negligible parts are obviously trivial for families of $L$-functions.
\end{remark}

The following proposition, though rather trivial, turns out to be useful.

\begin{proposition}
\label{coefneg}
For an asymptotically exact family of zeta functions $\{\zeta_k(s)\}$ we have $\lambda_f(\zeta_{k}(s))=\lambda_f(\zeta_{\bfe, k}(s)).$

\end{proposition}
\begin{proof}
This is an immediate corollary of proposition \ref{neglig}.
\end{proof}

The condition on a family to be asymptotically exact suffices for applications in the case of varieties over finite fields due to the positivity of coefficients $\Lambda_f.$ However, in general we will have to impose somewhat more restrictive conditions on the families.

\begin{definition}
We say that an asymptotically exact family of zeta or $L$-functions is asymptotically very exact if the series
$$\sum\limits_{f=1}^{\infty} |\lambda_f| q^{-\frac{f w_{\bfe}}{2}}$$
is convergent.
\end{definition}

\begin{example}
An obvious example of a family which is asymptotically exact but not very exact is given by the family of $L$-functions $L_k(s)=(1-q^{-s})^k.$ We have $\lambda_f=-1$ for any $f$ and the series $\sum\limits_{f=1}^{\infty} (-1)$ is clearly divergent.
\end{example}

\begin{proposition}
\label{zetaL}
Assume that we have an asymptotically exact family of zeta functions $\{\zeta_k(s)\}=\left\{\prod\limits_{i=0}^{w}L_{i k}(s)^{\epsilon_i}\right\}_{k=1\dots\infty}$, such that all the families $\{L_{i k}(s)\}$ are also asymptotically exact. Then, the family $\{\zeta_k(s)\}$ is asymptotically very exact if and only if the family $\{L_{w_{\bfe} k}(s)\}$ is asymptotically very exact.
\end{proposition}

\begin{proof}
This follows from proposition \ref{neglig} together with proposition \ref{coefneg}.
\end{proof}

In practice, this proposition means that the asymptotic behaviour of zeta functions for $\Re s>\frac{w_{\bfe}-1}{2}$ is essentially the same as that of their weight $w_{\bfe}$ parts. Thus, most asymptotic questions about zeta functions are reduced to the corresponding question about $L$-function.

\subsection{Examples}
As before we stick to three types of examples: curves over finite fields, varieties over finite fields and elliptic curves over function fields.

\begin{example}[Curves over finite fields]
Let $\{X_j\}$ be a family of curves over $\bbF_q.$ Recall (see \cite{TV97}) that an asymptotically exact family of curves was defined by Tsfasman and Vl\u adu\c t as such that the limits
\begin{equation}
\label{exacttsf}
\phi_f=\lim_{j\to\infty}\frac{\Phi_f(X_j)}{g_j}
\end{equation}
exist. This is equivalent to our definition since $\Lambda_f=N_f(X)=\sum\limits_{m\mid f}m\Phi_m.$ Note a little difference in the normalization of coefficients: in the case of curves we let $\lambda_{f}(\{X_j\})=\lim\limits_{j\to\infty}\frac{\Lambda_{f j}}{2g_j}$ since $2g_j$ is the degree of the corresponding polynomial in the numerator of $\zeta_{X_j}(s)$ and the authors of \cite{TV97} choose to consider simply $\lim\limits_{j\to\infty}\frac{\Lambda_{f j}}{g_j}.$

For any asymptotically exact family of zeta functions of curves the negligible part of $\zeta_X(s)$ is its denominator $(1-q^{-s})(1-q^{1-s})$ and the essential part is its numerator. Thus, zeta functions of curves asymptotically behave like $L$-functions. Any asymptotically exact family of curves is asymptotically very exact as shows the basic inequality from \cite{TV97} (see also corollary \ref{corineqL} below), which is in fact due to positivity of $\Lambda_f.$

\end{example}

\begin{example}[Varieties over finite fields]
In the case of varieties of fixed dimension $n$ over a finite field $\bbF_q$ we have an analogous notion of an asymptotically exact family \cite{LT}, namely we ask for the existence of the limits
$$\phi_f=\lim_{j\to\infty}\frac{\Phi_f(X_j)}{b(X_j)} \ \text { and } \ \delta_i=\beta_i=\lim\limits_{j\to \infty} \frac{b_i(X_j)}{b(X_j)},$$
where $b(X_j)=\sum\limits_{i=0}^{2n}b_i(X_j)$ is the sum of Betti numbers. Again this definition and our definition \ref{asexact} are equivalent.

In this case the factors $(1-q^{-s})$ and $(1-q^{n-s})$ of the denominator are also always negligible. However, we can have more negligible factors as the following example shows.

Take the product $C\times C,$ where $C$ is a curve of genus $g\to \infty.$ The dimension of the middle cohomology group $H^2$ grows as $g^2$ and $b_1=b_3=g$ (Kunneth formula). Thus $\zeta_{C\times C}(s)$ behaves like the inverse of an $L$-function.

If for an asymptotically exact family of varieties we have $w_{\bfe}=w-1=2n-1$ then it is asymptotically very exact as shows a form of the basic inequality \cite[(8.8)]{LT} (it actually gives that the series $\sum\limits_{f=1}^{\infty} \lambda_f q^{-f (n-1/2)}$ always converges), see also corollary \ref{corineqzeta} below.
\end{example}

\begin{example}[Elliptic curves over function fields]
In the last example we will be interested in two particular types of asymptotically exact families.

{\bf Asymptotically bad families.} Let us fix a function field $K=\bbF_q(X)$ and let us take the sequence of all pairwise non-isomorphic elliptic curves $E_i/K.$ We get a family of $L$-functions since $n_{E_i}\to \infty$. From (\ref{coeffell}) we deduce that $|\Lambda_f| \leq 2\left(\sum\limits_{d_v \mid f} d_v\right) q^{\frac{f}{2}}$ which is independent of $n_{E_i}.$ Thus, this family is asymptotically exact and asymptotically bad, i.~e. $\lambda_f = 0$ for any $f\geq 1.$ This will be the only fact important for our asymptotic considerations. There will be no difference in the treatment of this particular family or in that of any other asymptotically bad family of $L$-functions.

This family is considered in \cite{HP} in the connection with the generalized Brauer--Siegel theorem. The main result of that paper is the reduction of the statement about the behaviour of the order of the Tate--Shafarevich group and the regulator of elliptic curves over function fields to a statement about the values of their $L$-functions at $s = 1.$ See also \cite{Hin} for a similar problem treated in the number field case.

{\bf Base change.} Let us consider a family which is, in a sense, orthogonal to the previous one. Let $K=\bbF_q(X)$ be a function field and let $E/K$ be an elliptic curve. Let $f:\calE\to X$ be the corresponding elliptic surface. Consider a family of coverings of curves $X=X_0\leftarrow X_1 \dots\leftarrow X_i \leftarrow \dots$ and the family of elliptic surfaces $\calE_i,$ given by the base change:
\begin{equation*}
\CD \calE=\calE_0   @<<<  \calE_1  @<<< \dots @<<< \calE_i @<<< \dots \\
@VVfV  @VVV & &   @VVV \\
X=X_0   @<<<  X_1  @<<<  \dots  @<<<  X_i @<<< \dots.
\endCD
\end{equation*}

Let $\Phi_{v,f}(X_i)$ be the number of points on $X_i,$ lying above a closed point $v\in |X|,$ such that their residue fields have degree $f$ over $\bbF_v.$

\begin{lemma}
\label{veryexact}
The limits
$$\phi_{v, f} = \phi_{v, f} (\{X_i\})= \lim\limits_{i\to \infty} \frac{\Phi_{v, f}(X_i)}{g(X_i)}$$
always exist.
\end{lemma}
\begin{proof}
We will follow the proof of the similar statement for $\Phi_f$ from \cite[lemma 2.4]{TV02}. Let $K_2\supseteq K_1 \supseteq K$ be finite extension of function fields. From the Riemann--Hurwitz formula we deduce the inequality $g(K_2)-1\geq [K_2\colon K_1](g(K_1)-1),$
where $[K_2: K_1]$ is the degree of the corresponding extension. Now, if we fix $w$ a place of $K_1$ above $v$ and consider its decomposition $\{w_1, \dots, w_r\}$ in $K_2,$ then we have $\sum\limits_{i=1}^{r} \deg w_i \leq [K_2\colon K_1].$ Thus we get for any $n\geq 1$ the inequality
$\sum\limits_{f=1}^n f\Phi_{v,f}(K_2)\leq [K_2\colon K_1] \sum\limits_{f=1}^n f\Phi_{v,f}(K_1).$
Dividing we see that
$$\frac{\sum\limits_{f=1}^n f\Phi_{v,f}(K_2)}{g(K_2)-1} \leq \frac{\sum\limits_{f=1}^n f\Phi_{v,f}(K_1)}{g(K_1)-1}.$$
It follows that the sequence $\sum\limits_{f=1}^n \frac{f\Phi_{v,f}(X_i)}{g(X_i)-1}$ is non-increasing and non-negative for any fixed $n$ and so has a limit. Taking $n=1$ we see that $\phi_{v, 1}$ exists, taking $n=2$ we derive the existence of $\phi_{v,2}$ and so on.
\end{proof}

Let us remark that $\Phi_f(X_i)=\sum\limits_{m \deg v= f} \Phi_{v, m}(X_i),$ the sum being taken over all places $v$ of $K$ and the same equality holds for $\phi_f$ (in particular, the family $\{X_i\}$ is asymptotically exact).

For our family we can derive a concrete expression for the Dirichlet series coefficients of the logarithms of $L$-functions. Indeed, (\ref{coeffell}) gives us
\begin{equation}
 \label{coefKT}
 \Lambda_f=\sum_{m k d_v=f} m d_v \Phi_{v, m} (\alpha_v^{m k}+\bar{\alpha}_v^{m k}).
\end{equation}

\begin{lemma}
\label{conduct}
Let $E_i/K_i$ be a family of elliptic curves obtained by a base change and let $n_i=n_{E_i/K_i}$ be the degree of the conductor of $E_i/K_i.$ Then the ratio $\frac{n_i}{g_i}$ is bounded by a constant depending only on $E_0/K_0.$

If, furthermore, $\chr \bbF_q\neq 2, 3$ or the extensions $K_i/K_0$ are Galois for all $i$ then the limit $\nu=\lim\limits_{i\to\infty} \frac{n_i}{g_i}$ exists.
\end{lemma}
\begin{proof}
The proof basically consists of looking at the definition of the conductor and applying the same method as in the proof of lemma \ref{veryexact}. Recall, that if $E/K$ is an elliptic curve over a local field $K$, $T_l(E)$ is its Tate module, $l\neq \chr \bbF_q,$ $V_l(E)=T_l(E)\otimes \bbQ_l,$  $I(\bar{K}/K)$ is the inertia subgroup of $\Gal(\bar{K}/K),$ then the tame part of the conductor is defined as
$$\varepsilon(E/K)=\dim_{\bbQ_l}(V_l(E)/V_l(E)^{I(\bar{K}/K)}).$$
It is easily seen to be non increasing in extensions of $K,$ moreover it is known that $0\leq \varepsilon(E/K)\leq 2$ (see \cite[Chap. IV, \S 10]{Sil}).

If we let $L=K(E[l]),$ $\gamma_i(L/K)=|G_i(L/K)|,$ where $G_i(L/K)$ is the $i^{\,\text{th}}$ ramification group of $L/K,$ then the wild part of the conductor is defined as
$$\delta(E/K)=\sum_{i=1}^{\infty}\frac{\gamma_i(L/K)}{\gamma_0(L/K)}\dim_{\bbF_l} (E[l]/E[l]^{G_i(L/K)}).$$
One can prove \cite[Chap. IV, \S 10]{Sil} that $\delta(E/K)$ is zero unless the characteristic of the residue field of $K$ is equal to $2$ or $3.$ In any case, the definition shows that $\delta(E/M)$ can take only finitely many values if we fix $E$ and let vary the extension $M/K.$

The exponent of the conductor of $E$ over the local field $K$ is defined to be $f(E/K)=\varepsilon(E/K)+\delta(E/K).$ For an elliptic curve $E$ over a global field $K$ the $v$-exponent of the conductor is taken to be $n_v(E)=f(E/K_v),$ where $K_v$ is the completion of $K$ at $v.$

From the previous discussion we see that the ratio $\frac{n_i}{g_i}$ is bounded. If, furthermore, $\chr \bbF_q\neq 2, 3,$ then an argument similar to the one used in the proof of lemma \ref{veryexact} together with the fact that $n_w(E) \leq n_v(E)$ if $w$ lies above $v$ in an extension of fields gives us that the sequence $\frac{n_i}{g_i}$ is non-increasing and so it has a limit $\nu=\nu(\{E_i/K_i\}).$

In the case of Galois extensions we notice that $n_w(E)$ must stabilize in a tower and all the $n_w(E)$ are equal for $w$ over a fixed place $v$. Thus the previous argument is applicable once again.
\end{proof}

Now we can prove the following important proposition:

\begin{proposition}
\label{veryexactell}
Any family of elliptic curves obtained by a base change contains an asymptotically very exact subfamily. If, furthermore, $\chr \bbF_q\neq 2, 3$ or the extensions $K_i/K_0$ are Galois for all $i$ then it is itself asymptotically very exact.
\end{proposition}
\begin{proof}
Recall that for each $E_i/K_i$ the degree of the corresponding $L$-function is $n_i+4g_i-4.$ It follows from the previous lemma that it is enough to prove the existence of the limits $\tilde{\lambda}_f=\lim\limits_{i\to\infty}\frac{\Lambda_f(E_i)}{g_i}$ and the convergence of the series $\sum\limits_{f=1}^{\infty}|\tilde{\lambda}_f| q^{-f}.$

The first statement is a direct corollary of lemma \ref{veryexact} and (\ref{coefKT}). As for the second statement, we have the following bound:
\begin{equation*}
 |\Lambda_f|\leq 2\sum_{m k d_v=f} m d_v \Phi_{v, m} q^{\frac{f}{2}}=2\sum_{l k =f} l \Phi_{l} q^{\frac{f}{2}}=2 N_f q^{\frac{f}{2}}.
\end{equation*}
Now, the convergence of the series $\sum\limits_{f=1}^{\infty} \nu_f q^{-\frac{f}{2}}$  with $\nu_f=\lim\limits_{i\to \infty}\frac{N_f(X_i)}{g_i}$ is a consequence of the basic inequality for zeta-functions of curves (\cite[corollary 1]{Tsf} or example \ref{CurveIneq}).
\end{proof}

\begin{remark}
It would be nice to know whether the statement of the previous proposition holds without any additional assumptions, i. e. whether a family obtained by a base change is always asymptotically very exact. This depends on lemma \ref{conduct}, which do not know how to prove in general.
\end{remark}

The family of elliptic curves obtained by the base change was studied in \cite{KT} again in the attempts to obtain a generalization of the Brauer--Siegel theorem to this case. Kunyavskii and Tsfasman formulate a conjecture on the asymptotic behaviour of  the order of the Tate--Shafarevich group and the regulator in such families (see conjecture \ref{KunTsfBS} below). They also treat the case of constant elliptic curves in more detail. Unfortunately, the proof of the main theorem \cite[theorem 2.1]{KT} given there is not absolutely flawless (the change of limits remains to be justified, which seems to be very difficult if not inaccessible at present).

\begin{remark}
If, for a moment, we turn our attention to general families of elliptic surfaces the following natural question arises:

\begin{question}
Is it true that any family of elliptic surfaces contains an asymptotically very exact subfamily?
\end{question}

The fact that it is true for two ``orthogonal'' cases makes us believe that this property might hold in general.
\end{remark}
\end{example}

\section{Basic inequalities}
\label{sectineq}
In this section we finally start carrying out our program to generalize asymptotic results from the case of curves over finite fields to the case of general zeta and $L$-functions. We will start with the case of $L$-functions, where a little more can be said. Next, we will prove a weaker result in the case of zeta functions.

\subsection{Basic inequality for $L$-functions}
Our goal here is to prove the following theorem, generalizing the basic inequality from \cite{Tsf}.
\begin{theorem}
\label{basicineqL}
Assume we have an asymptotically exact family $\{L_k(s)\}$ of $L$-functions of weight $w$ or an asymptotically exact family of zeta functions $\{\zeta_i(s)\}$ with $\zeta_{\bfe, i}(s)$ being an $L$-function of weight $w$ for any $i.$ Then for any $b\in \bbN$ the following inequality holds:
\begin{equation}
\label{basicineqformula}
\sum_{j=1}^b\left(1-\frac{j}{b+1}\right)\lambda_j q^{-\frac{w j}{2}} \leq \frac{1}{2}.
\end{equation}
\end{theorem}
\begin{proof}
Using proposition \ref{coefneg} one immediately sees that it is enough to prove the statement of the theorem for $L$-functions.

As in the proof for curves our main tool will be the so called Drinfeld inequality. We take an $L$-function $L(s)$ and let $\alpha_i=q^{\frac{w}{2}}\rho_i,$ where $\rho_i$ are the roots of $\calL(u),$ so that $|\alpha_i|=1.$
For any $\alpha_i$ we have
$$0\leq |\alpha_i^b+\alpha_i^{b-1}+\dots+1|^2=(b+1)+\sum_{j=1}^b (b+1-j)(\alpha_i^j+\alpha_i^{-j}).$$
Thus $b+1\geq - \sum\limits_{j=1}^b (b+1-j)(\alpha_i^j+\alpha_i^{-j}).$ We sum the inequalities for $i=1,\dots, d.$ Since the coefficients of $\calL(u)$ are real we note that $\sum\limits_{i=1}^d \alpha_i^j=\sum\limits_{i=1}^d \alpha_i^{-j}.$ From (\ref{zerosums}) we see that $\Lambda_j=-q^{w j}\sum\limits_{i=1}^d \rho_i^j.$ Putting it together we get:
$$d(b+1)\geq 2\sum\limits_{j=1}^b (b+1-j)\Lambda_j q^{-\frac{w j}{2}}.$$
Now, we let vary $L_k(s)$ so that $d_k\to \infty$ and obtain the stated inequality.
\end{proof}

Unfortunately, we are unable to say anything more in general without the knowledge of some additional properties of $\lambda_j.$ However, the next corollary shows that sometimes we can do better.
\begin{corollary}
\label{corineqL}
If a family $\{L_k(s)\}$ is asymptotically exact then
$$\sum_{j=1}^{\infty} \lambda_j q^{-\frac{w j}{2}} \leq \frac{1}{2},$$
provided one of the following conditions holds:
\begin{enumerate}
\item either it is asymptotically very exact or
\item $\lambda_j\geq 0$ for any $j.$
\end{enumerate}
\end{corollary}

\begin{proof}
To prove the statement of the corollary under the first assumption we choose an $\varepsilon>0$ and $b'\in \bbN$ such that the sum $\sum\limits_{j=b'+1}^{\infty} |\lambda_j| q^{-\frac{w j}{2}} < \varepsilon.$ Then we choose $b''$ such that $\frac{b'}{b''+1}<\varepsilon.$ Now we apply the inequality from theorem \ref{basicineqL} with $b=b''.$ We get:
\begin{multline*}
\frac{1}{2}\geq\sum_{j=1}^{b''}\left(1-\frac{j}{b''+1}\right)\lambda_j q^{-\frac{w j}{2}} \geq \sum_{j=1}^{b'}\left(1-\frac{j}{b''+1}\right)\lambda_j q^{-\frac{w j}{2}} +\\+
\sum_{j=b'+1}^{b''}\left(1-\frac{j}{b''+1}\right)\lambda_j q^{-\frac{w j}{2}} \geq (1-\varepsilon)\sum_{j=1}^{\infty}\lambda_j q^{-\frac{w j}{2}} -2\varepsilon.
\end{multline*}
So the first part of the corollary is true.

To prove the statement under the second condition we use the same trick. We take $b'\in \bbN$ such that $\frac{b}{b'+1}<\varepsilon.$ Then we apply theorem  \ref{basicineqL} with $b=b'$ and notice that the sum only decreases when we drop the part $\sum\limits_{j=b+1}^{b'}\left(1-\frac{j}{b'+1}\right)\lambda_j q^{-\frac{w j}{2}}$ since $\lambda_j\geq 0.$ This gives the second part of the corollary.
\end{proof}

\begin{corollary}
Any asymptotically exact family of $L$-functions, satisfying $\lambda_j\geq 0$ for any $j,$ is asymptotically very exact.
\end{corollary}

\begin{remark}
The statements of both of the corollaries are obviously still true if one assumes that $\lambda_j\geq 0$ for all but a finite number of $j\in \bbN.$
\end{remark}

\begin{remark}
The methods from the section \ref{sectzero} allow us to prove a little stronger statement for asymptotically very exact families. See remark \ref{familyineq} for details.
\end{remark}

\subsection{Basic inequality for zeta functions}
We have noticed before that even in the case of $L$-functions we do not get complete results unless we assume that our family is asymptotically very exact or all the coefficients $\lambda_f$ are positive. While working with zeta functions we face the same problem. However, we will deal with it in a different way for no general lower bound on the sums of the type (\ref{basicineqformula}) seems to be available and such a lower bound would definitely be necessary since zeta functions are products of $L$-functions both in positive and in negative powers.

\begin{theorem}
\label{basicineqzeta}
Let $\{\zeta_k(s)\}$ be an asymptotically exact family of zeta functions. Then for any real $s$ with $\frac{w_{\bfe}}{2}<s<\frac{w_{\bfe}+1}{2}$ we have:
\begin{equation*}
-\sum_{i=0}^{w_{\bfe}}\frac{\delta_i}{q^{s-i/2}-\epsilon_i} \leq \sum_{j=1}^{\infty} \lambda_j q^{-s j} \leq \sum_{i=0}^{w_{\bfe}}\frac{\delta_i}{q^{s-i/2}+\epsilon_i}.
\end{equation*}
\end{theorem}

\begin{proof}
First of all, proposition \ref{coefneg} implies that it is enough to prove the inequality in the case when $\zeta_k(s)=\zeta_{\bfe, k}(s)$ and $w=w_{\bfe}.$

Let us write the Stark formula from proposition \ref{Stark}:
\begin{equation*}
\frac{1}{\log q}\frac{\zeta'(s)}{\zeta(s)}= \sum_{i=0}^{w}\epsilon_i\sum_{j=1}^{d_i}\frac{1}{q^{s}{\rho_{i j}}-1}.
\end{equation*}
We notice that all the terms are real for real $s$ and
\begin{equation*}
R(r, \theta)=\Re \frac{re^{i\theta}}{1-re^{i\theta}}= \frac{r\cos \theta-r^2}{1-2r\cos\theta+r^2}.
\end{equation*}
Applying this relation we see that
$$\frac{1}{\log q}\frac{\zeta'(s)}{\zeta(s)}= \sum_{i=0}^{w}\epsilon_i\sum_{j=1}^{d_i}R(q^{i/2-s},\theta_{i j}),$$
where  $\rho_{k j}=q^{-\frac{k}{2}}e^{i\theta_{k j}}.$

For $0<r<1$ we have the bounds on $R(r, \theta):$
$$-\frac{r}{1+r}\leq R(r, \theta) \leq \frac{r}{1-r}.$$
From this we deduce that for $s$ with $\frac{w}{2}<s<\frac{w+1}{2}$ the following inequality holds
\begin{equation}
\label{bound}
-\sum_{i=0}^{w}\frac{d_i}{q^{s-i/2}-\epsilon_i} \leq \frac{-1}{\log q}\frac{\zeta'(s)}{\zeta(s)} \leq \sum_{i=0}^{w}\frac{d_i}{q^{s-i/2}+\epsilon_i}.
\end{equation}

The next step is to use theorem \ref{BSGen}. For any $s$ in the interval $\left(\frac{w}{2}, \frac{w+1}{2}\right)$ it gives that
$$\lim_{k\to \infty}\frac{-1}{\tilde{d}_k \log q}\cdot\frac{\zeta_k'(s)}{\zeta_k(s)}=\sum_{j=1}^{\infty} \lambda_j q^{-\frac{s j}{2}}.$$
Dividing (\ref{bound}) by $\tilde{d}_k,$ passing to the limit and using the previous equality we get the statement of the theorem.
\end{proof}

\begin{corollary}
\label{corineqzeta}
\begin{enumerate}
\item If $\epsilon_{w_{\bfe}}=1$ and either the family is asymptotically very exact or $\lambda_j\geq 0$ for all $j$ then
$$\sum_{j=1}^{\infty} \lambda_j q^{-\frac{w_{\bfe} j}{2}} \leq \sum_{i=0}^{w_{\bfe}}\frac{\delta_i}{q^{(w_{\bfe}-i)/2}+\epsilon_i}$$
\item If $\epsilon_{w_{\bfe}}=-1$ and either the family is asymptotically very exact or $\lambda_j\leq 0$ for all $j$ then
$$-\sum_{i=0}^{w_{\bfe}}\frac{\delta_i}{q^{(w_{\bfe}-i)/2}-\epsilon_i} \leq \sum_{j=1}^{\infty} \lambda_j q^{-\frac{w_{\bfe} j}{2}}.$$
\end{enumerate}
\end{corollary}

\begin{proof}
Let us suppose that $\epsilon_{w_{\bfe}}=1$ (the other case is treated similarly). For an asymptotically very exact family for any $\varepsilon >0$ we can choose $N>0$ such that $\sum\limits_{j > N}^{\infty} |\lambda_j| q^{-\frac{w_{\bfe} j}{2}} < \varepsilon.$ Thus both for a very exact family and for a family with $\lambda_j\geq 0$ for all $j$ we have
$$\sum_{j=1}^{N} \lambda_j q^{-s j} \leq \sum_{i=0}^{w_{\bfe}}\frac{\delta_i}{q^{s-i/2}+\epsilon_i}+\varepsilon$$
for any real $s$ with $\frac{w_{\bfe}}{2}<s<\frac{w_{\bfe}+1}{2}.$ Passing to the limit when $s \to \frac{w_{\bfe}}{2}$ we get the statement of the corollary.
\end{proof}

\begin{corollary}
\label{zetaasexact}
Any asymptotically exact family, such that $\epsilon_{w_{\bfe}}\sign(\lambda_j)=1$ for any $j,$ is asymptotically very exact.
\end{corollary}

\begin{remark}
Though the corollary \ref{corineqzeta} implies the corollary \ref{corineqL}, the basic inequality for $L$-functions given by theorem \ref{basicineqL} is different from the one obtained by application of theorem \ref{basicineqzeta}.
\end{remark}

\subsection{Examples}

\begin{example}[Curves over finite fields]
\label{CurveIneq}
For curves over finite fields we obtain once again the classical basic inequality from \cite{Tsf}:
\begin{equation*}
\sum_{j=1}^{\infty} 2\lambda_j q^{-\frac{j}{2}} =\sum_{m=1}^{\infty} \frac{m\phi_m}{q^{m/2}-1}\leq 1.
\end{equation*}
Of course, this is not an interesting example for us, since we used this inequality as our initial motivation.
\end{example}

\begin{example}[Varieties over finite fields]
In a similar way, for varieties over finite fields we get the inequality from \cite[(8.8)]{LT}:
\begin{equation*}
\sum_{m=1}^{\infty} \frac{m\phi_m}{q^{(2d-1)m/2}-1}\leq (q^{(2d-1)/2}-1)\left(\frac{\beta_1}{2}+\sum_{2 \mid i}\frac{\beta_i}{q^{(i-1)/2}+1}+\sum_{2 \nmid i} \frac{\beta_i}{q^{(i-1)/2}-1}\right).
\end{equation*}

With more efforts one can reprove most (if not all) of the inequalities from \cite[(8.8)]{LT} in our general context of zeta functions, since the main tools used in \cite{LT} are the explicit formulae. However, we do not do it here as for the moment we are unable see any applications it might have to particular examples of zeta functions.
\end{example}

\begin{example}[Elliptic curves over function fields]

The case of asymptotically bad families is trivial: we do not obtain any interesting results here since all $\lambda_j=0$.

Let us consider the base change case. Let us take an asymptotically very exact family of elliptic curves obtained by a base change (by proposition \ref{veryexactell} any family obtained by a base change is asymptotically very exact, provided $\chr\bbF_q \neq 2, 3$). We can apply corollary \ref{corineqL} to obtain that $\sum\limits_{j=1}^{\infty} \lambda_j q^{-j/2} \leq \frac{1}{2}.$ Using (\ref{coefKT}), one can rewrite it using $\phi_{v, m}$ as follows:
\begin{equation*}
\sum_{v, m}\frac{m d_v \phi_{v, m}(\alpha_v^m+\bar{\alpha}_v^m)q^{-m d_v}}{1-(\alpha_v^m+\bar{\alpha}_v^m)q^{-m d_v}}\leq \frac{\nu+4}{2}
\end{equation*}
(here $\nu=\lim\limits_{i\to\infty}\frac{n_{E_i/K_i}}{g_{K_i}}$).
\end{example}

\section{Brauer--Siegel type results}
\label{sectBS}

\subsection{Limit zeta functions and the Brauer--Siegel theorem}

Our approach to the Brauer--Siegel type results will be based on limit zeta functions.

\begin{definition}
Let $\{\zeta_k(s)\}$ be an asymptotically exact family of zeta functions. Then the corresponding limit zeta function is defined as
$$\zeta_{\bflim}(s)=\exp\left(\sum_{f=1}^{\infty} \frac{\lambda_f}{f}\,q^{-f s}\right).$$
\end{definition}

\begin{remark}
If $\zeta_k(s)=\zeta_{f_k}(s)$ are associated to some arithmetic or geometric objects $f_k$ we will denote the limit zeta function simply by $\zeta_{\{f_k\}}(s).$
\end{remark}
\begin{remark}
The basic inequality from theorem \ref{basicineqzeta} can be reformulated in terms of $\zeta_{\bflim}(s)$ as
\begin{equation*}
-\sum_{i=0}^{w_{\bfe}}\frac{\delta_i}{q^{s-i/2}-\epsilon_i} \leq -\frac{1}{\log q}\, \frac{\zeta'_{\bflim}(s)}{\zeta_{\bflim}(s)} \leq \sum_{i=0}^{w_{\bfe}}\frac{\delta_i}{q^{s-i/2}+\epsilon_i}.
\end{equation*}
\end{remark}

Here are the first elementary properties of limit zeta functions:

\begin{proposition}
\label{convergencelim}
\begin{enumerate}
\item For an asymptotically exact family of zeta functions $\{\zeta_k(s)\}$ the series for $\log\zeta_{\bflim}(s)$ is absolutely and uniformly convergent on compacts in the domain $\Re s > \frac{w_{\bfe}}{2},$ defining an analytic function there.
\item If a family is asymptotically very exact then $\zeta_{\bflim}(s)$ is continuous for $\Re s \geq \frac{w_{\bfe}}{2}.$
\item If for a family we have $\lambda_j\geq 0$ for any $j$ and $\epsilon_{w_{\bfe}}=1,$ then the series for $\log\zeta_{\bflim}(s)$ is absolutely and uniformly convergent in the domain $\Re s\geq \frac{w_{\bfe}}{2}-\delta$ for some $\delta > 0.$
\end{enumerate}
\end{proposition}

\begin{proof}
The first part of the proposition obviously follows from proposition \ref{neglig} together with proposition \ref{coefneg}.

By the definition of an asymptotically very exact family, the series for $\log \zeta_{\bflim}(s)$ is uniformly and absolutely convergent for $\Re s \geq \frac{w_{\bfe}}{2}$ so defines a continuous function in this domain. Thus the second part is proven.

To get the third part we apply corollary \ref{zetaasexact} to see that our family is asymptotically very exact. Then we use a well known fact that the domain of convergence of a Dirichlet series with non-negative coefficients is an open half-plane $\Re s > \sigma.$
\end{proof}

It is important to see to which extent limit zeta functions are the limits of the corresponding zeta functions over finite fields. The question is answered by the generalized Brauer--Siegel theorem. Before stating it let us give one more definition.

\begin{definition}
For an asymptotically exact family of zeta functions $\{\zeta_k(s)\}$ we call the limit $\lim\limits_{k\to \infty}\frac{\log\zeta_{k}(s)}{\tilde{d}_k}$ the Brauer--Siegel ratio of this family.
\end{definition}

\begin{theorem}[The generalized Brauer--Siegel theorem]
\label{BSGen}
For any asymptotically exact family of zeta functions $\{\zeta_k(s)\}$ and any $s$ with $\Re s > \frac{w_{\bfe}}{2}$ we have
\begin{equation*}
\lim_{k\to \infty}\frac{\log\zeta_{\bfe, k}(s)}{\tilde{d}_k}=\log\zeta_{\bflim}(s).
\end{equation*}
If, moreover, $2 \Re s \not \in \bbZ,$ then
\begin{equation*}
\lim_{k\to \infty}\frac{\log\zeta_{k}(s)}{\tilde{d}_k}=\log\zeta_{\bflim}(s).
\end{equation*}
The convergence is uniform in any domain $\frac{w_{\bfe}}{2}+\varepsilon<\Re s < \frac{w_{\bfe}+1}{2}-\varepsilon,$ $\varepsilon\in\left(0, \frac{1}{2}\right).$
\end{theorem}
\begin{proof}
To get the first statement we apply proposition \ref{coefneg} and exchange the limit when $k\to\infty$ and the summation, which is legitimate since the series in question are absolutely and uniformly convergent in a small (but fixed) neighbourhood of $s$.

To get the second statement we apply proposition \ref{neglig}, which gives us:
$$\lim_{k\to \infty}\frac{\log\zeta_{\bfn,k}(s)}{\tilde{d}_k}=0.$$
Now the second part of the theorem follows from the first.
\end{proof}

\begin{remark}
It might be unclear, why we call such a statement the Brauer--Siegel theorem. We will see below in subsection \ref{exBS} that the above theorem indeed implies a natural analogue of the Brauer--Siegel theorem for curves and varieties over finite fields. It is quite remarkable that the proof of theorem \ref{BSGen} is very easy (say, compared to the one in \cite{TV97}) once one gives proper definitions.
\end{remark}

\begin{remark}
Let us sketch another way to prove the generalized Brauer--Siegel theorem. It might seem unnecessarily complicated but it has the advantage of being applicable in the number field case when we no longer have the convergence of $\log L_k(s)$ for $\Re s > \frac{w}{2}.$ We will deal with $L$-functions to simplify the notation. The main idea is to prove using Stark formula (proposition \ref{Stark} in the case of $L$-functions over finite fields) that $\frac{L'_k(s)}{L_k(s)}\leq C(\varepsilon) d_k$ for any $s$ with $\Re s \geq \frac{w}{2}+ \epsilon.$ Then we apply the Vitali theorem from complex analysis, which states that for a sequence of bounded holomorphic functions in a domain $\calD$ it is enough to check the convergence at a set of points in $\calD$ with a limit point in $\calD.$ This method is applied to Dedekind zeta functions in \cite{Z3}.
\end{remark}

\begin{remark}
\label{difficulties}
It is natural to ask, what is the behaviour of limit zeta or $L$-functions for $\Re s \leq \frac{w_{\bfe}}{2}.$ Unfortunately nice properties of $L$-functions such as the functional equation or the Riemann hypothesis do not hold for $L_{\bflim}(s).$ This can be seen already for families of zeta functions of curves. The point is that the behaviour of $L_{\bflim}(s)$ might considerably differ from that of $\lim\limits_{k\to \infty}\frac{\log L_{k}(s)}{d_k}$ when we pass the critical line.
\end{remark}

\subsection{Behaviour at the central point}
\label{centralpoint}

It seems reasonable to ask, what is the relation between limit zeta functions and the limits of zeta functions over finite fields on the critical line (that is for $\Re s = \frac{w_{\bfe}}{2}).$ This relation seems to be rather complicated. For example, one can prove that the limit $\lim\limits_{k\to \infty}\frac{1}{\tilde{d}_k}\frac{\zeta'_{k}(1/2)}{\zeta_k(1/2)}$ is always $1$ in families of curves (this can be seen from the functional equation), which is definitely not true for the value $\frac{\zeta'_{\bflim}(1/2)}{\zeta_{\bflim}(1/2)}.$

However, the knowledge of this relation is important for some arithmetic problems (see the example of elliptic surfaces in the next subsection). The general feeling is that for ``most'' families the statement of the generalized Brauer--Siegel theorem holds for $s=\frac{w_{\bfe}}{2}.$ There are very few cases when we know it (see section \ref{sectquestions} for a discussion) and we, actually, can not even formulate this statement as a conjecture, since it is not clear what conditions on $L$-functions we should impose.

Still, in general one can prove the ``easy'' inequality. The term is borrowed from the classical Brauer--Siegel theorem from the number field case, where the upper bound is known unconditionally (and is easy to prove) and the lower bound is not proven in general (one has to assume either GRH or a certain normality condition on the number fields in question). This analogy does not go too far though for in the classical Brauer--Siegel theorem we work far from the critical line and here we study the behaviour of zeta functions on the critical line itself.

Let $\{\zeta_k(s)\}$ be an asymptotically exact family of zeta functions. We define $r_k$ and $c_k$ using the Taylor series expansion $$\zeta_k(s)=c_k\left(s-\frac{w_{\bfe}}{2}\right)^{r_k}+O\left(\left(s-\frac{w_{\bfe}}{2}\right)^{r_k+1}\right).$$

\begin{theorem}
\label{upperboundBS}
For an asymptotically very exact family of zeta functions $\{\zeta_k(s)\}$ such that $\epsilon_{w_{\bfe}}=1$ we have:
\begin{equation*}
\lim_{k\to \infty}\frac{\log |c_k|}{\tilde{d}_k} \leq\log\zeta_{\bflim}\left(\frac{w_{\bfe}}{2}\right).
\end{equation*}
\end{theorem}

\begin{proof}
Replacing the family $\{\zeta_k(s)\}$ by the family $\{\zeta_{\bfe, k}(s)\}$ we can assume that $w=w_{\bfe}.$

Let us write
$$\zeta_k(s)=c_k \left(s-\frac{w}{2}\right)^{r_k} F_k(s),$$
where $F_k(s)$ is an analytic function in the neighbourhood of $s=\frac{w}{2}$ such that $F_k\left(\frac{w}{2}\right)=1.$
Let us put $s=\frac{w}{2}+\theta,$ where $\theta>0$ is a small positive real number. We have
\begin{equation*}
\frac{\log \zeta_k(\frac{w}{2}+\theta)}{\tilde{d}_k}=\frac{\log c_k}{\tilde{d}_k}+r_k\frac{\log \theta}{\tilde{d}_k}+\frac{\log F_k(\frac{w}{2}+\theta)}{\tilde{d}_k}.
\end{equation*}
To prove the theorem we will construct a sequence $\theta_k$ such that
\begin{enumerate}
\item[(1)] $\frac{1}{\tilde{d}_k}\log \zeta_k\left(\frac{w}{2}+\theta_k\right) \to \log \zeta_{\bflim}\left(\frac{w}{2}\right);$
\item[(2)] $\frac{r_k}{\tilde{d}_k}\,\log \theta_k \to 0;$
\item[(3)] $\liminf \frac{1}{\tilde{d}_k}\log F_k\left(\frac{w}{2}+\theta_k\right)\geq 0.$
\end{enumerate}

For each natural number $N$ we choose $\theta(N)$ a decreasing sequence such that
\begin{equation*}
\left|\zeta_{\bflim}\left(\frac{w}{2}\right)-\zeta_{\bflim}\left(\frac{w}{2}+\theta(N)\right)\right| < \frac{1}{2N}.
\end{equation*}
This is possible since $\zeta_{\bflim}(s)$ is continuous for $\Re s \geq \frac{w}{2}$ by proposition \ref{convergencelim}. Next, we choose a sequence $k'(N)$ with the property:
\begin{equation*}
\left|\frac{1}{d_{k}}\log \zeta_{k}\left(\frac{w}{2}+\theta \right) - \log \zeta_{\bflim}\left(\frac{w}{2}+\theta\right)\right|< \frac{1}{2N}
\end{equation*}
for any $\theta\in[\theta(N+1), \theta(N)]$ and any $k\geq k'(N).$ This is possible by theorem \ref{BSGen}. Then we choose $k''(N)$ such that
\begin{equation*}
\frac{-r_k\log \theta(N+1)}{\tilde{d}_k}\leq \frac{\theta(N)}{N}
\end{equation*}
for any $k\geq k''(N),$ which can be done thanks to corollary \ref{rankbound} that gives us for an asymptotically very exact family $\frac{r_k}{d_k}\to 0.$ Finally, we choose an increasing sequence $k(N)$ such that $k(N)\geq \max(k'(N), k''(N))$ for any $N.$

Now, if we define $N=N(k)$ by the inequality $k(N)\leq k \leq k(N+1)$ and let $\theta_k=\theta(N(k)),$ then from the conditions imposed on $\theta_k$ we automatically get (1) and (2). The delicate point is (3). We apply the Stark formula from proposition \ref{Stark} to get an estimate on $\left(\log F_k\left(\frac{w}{2}+\theta\right)\right)':$

\begin{multline*}
\frac{1}{\tilde{d}_k}\left(\log \zeta_k\left(\frac{w}{2}+\theta\right)-r_k\log \theta\right)'=-\frac{\log q}{2\tilde{d}_k}\,\sum_{i=0}^{w}\epsilon_i d_i +\\
+\frac{1}{\tilde{d}_k}\,\sum_{i=0}^{w-1}\epsilon_i\sum_{L_i(\theta_{i j})=0}\frac{1}{\frac{w}{2}+\theta-\theta_{i j}}+\frac{1}{\tilde{d}_k}\sum_{L_w(\theta_{w j})=0, \theta_{w j}\neq \frac{w}{2}} \frac{1}{\frac{w}{2}+\theta-\theta_{w j}}.
\end{multline*}
The first term on the right hand side is clearly bounded by $-\log q$ from below. The first sum involving $L$-functions is also bounded by a constant $C_1$ as can be seen applying the Stark formula to individual $L$-functions and then using proposition \ref{neglig}. The last sum is non-negative. Thus, we see that $\frac{1}{\tilde{d}_k}\left(\log F_k\left(\frac{w}{2}+\theta\right)\right)'\geq C$ for any small enough $\theta.$ From this and from the fact that $F_k\left(\frac{w}{2}\right)=1$ we deduce that
$$\frac{1}{\tilde{d}_k}\log F_k\left(\frac{w}{2}+\theta_k\right)\geq C\theta_k\to 0.$$
This proves (3) as well as the theorem.
\end{proof}

\begin{remark}
In the case when $\epsilon_{w_{\bfe}}=-1$ we get an analogous statement with the opposite inequality.
\end{remark}

\begin{remark}
The proof of the theorem shows the importance of ``low'' zeroes of zeta functions (that is zeroes close to $s=\frac{w}{2}$) in the study of the Brauer--Siegel ratio at $s =\frac{w}{2}.$ The lack of control of these zeroes is the reason why we can not prove a lower bound on $\lim\limits_{k\to\infty}\frac{\log |c_k|}{\tilde{d}_k}.$
\end{remark}

\begin{remark}
If we restrict our attention to $L$-functions with integral coefficients (i. e. such that $\calL(u)$ has integral coefficients), then we can see that the ratio $\frac{\log |c_k|}{\tilde{d}_k}$ is bounded from below by $-w\log q,$  at least for even $w.$ This follows from a simple observation that if a polynomial with integral coefficients has a non-zero positive value at an integer point then this value is greater then or equal to one. One may ask whether there is a lower bound for arbitrary $w$ and whether anything similar holds in the number field case.
\end{remark}

\subsection{Examples}
\label{exBS}

\begin{example}[Curves over finite fields]

First of all, let us show that the generalized Brauer--Siegel theorem \ref{BSGen} implies the standard Brauer--Siegel theorem for curves over finite fields from \cite{TV97}.

Let $h_X$ be the number of $\bbF_q$-rational pints on the Jacobian of $X$, i. e.  $h_X=|\Pic^0_{\bbF_q}(X)|.$

\begin{corollary}
\label{BSCurves}
For an asymptotically exact family of curves $\{X_i\}$ over a finite field $\bbF_q$ we have:
\begin{equation}
\label{eqBScurves}
\lim_{i\to\infty}\frac{\log h_{X_i}}{g_i}= \log q + \sum_{f=1}^{\infty} \phi_f \log \frac{q^f}{q^f-1}.
\end{equation}
\end{corollary}

\begin{proof}
It is well known (cf. \cite{TVN}) that for a curve $X$ the number $h_X$ can be expressed as $h_X=\calL_X(1),$ where $\calL_X(u)$ is the numerator of the zeta function of $X.$ Using the functional equation for $\zeta_X(s)$ we see that this expression is equal to $L_X(0)=L_X(1)+g\log q.$

The right hand side of (\ref{eqBScurves}) can be written as $\log q + 2\log \zeta_{\{X_i\}}(1),$ where $\zeta_{\{X_i\}}(s)$ is the limit zeta function (the factor $2$ appears from the definition of $\log \zeta_{\{X_i\}}(s),$ in which we divide by $2g$ and not by $g$). Thus, it is enough to prove that
$$\lim_{i\to\infty}\frac{\log L_{X_i}(1)}{2g_i}=\log \zeta_{\{X_i\}}(1).$$
This follows immediately from the first equality of theorem \ref{BSGen}.
\end{proof}

Using nearly the same proof we can obtain one more statement about the asymptotic behaviour of invariants of function fields. To formulate it we will need to define the so called Euler--Kronecker constants of a curve $X$ (see \cite{Ih}):

\begin{definition}
Let $X$ be a curve over a finite field $\bbF_q$ and let
$$\frac{\zeta_X'(s)}{\zeta_X(s)}=-(s-1)^{-1} +\gamma_X^0+\gamma_X^1(s-1)+\gamma_X^2(s-1)^2+\dots$$
be the Taylor series expansion of $\frac{\zeta_X'(s)}{\zeta_X(s)}$ at $s=1.$ Then $\gamma_X=\gamma_X^0$ is called the Euler--Kronecker constant of $X$ and $\gamma_X^k,$ $k\geq 1$ are be called the higher Euler-Kronecker constants.
\end{definition}

We also define the asymptotic Euler-Kronecker constants $\gamma_{\{X_i\}}^k$ from:
$$\frac{\zeta_{\{X_i\}}'(s)}{\zeta_{\{X_i\}}(s)}=\gamma_{\{X_i\}}^0+\gamma_{\{X_i\}}^1(s-1)+\gamma_{\{X_i\}}^2(s-1)^2+\dots$$
($\zeta_{\{X_i\}}(s)$ is holomorphic and non-zero at $s=1$ so its logarithmic derivative has no pole at this point).

The following result generalizes theorem 2 from \cite{Ih}:

\begin{corollary}
\label{EulerKronecker}
For an asymptotically exact family of curves $\{X_i\}$ we have
$$\lim\limits_{i\to \infty}\frac{\gamma_{X_i}^k}{g_i}= \gamma_{\{X_i\}}^k$$
for any non-negative integer $k.$ In particular,
$$\lim\limits_{i\to \infty}\frac{\gamma_{X_i}}{g_i}=-\sum_{f=1}^{\infty} \frac{\phi_f f \log q}{q^f-1}.$$
\end{corollary}
\begin{proof}.
We apply the first equality from theorem \ref{BSGen}. Using the explicit expression for the negligible part of zetas as $(1-q^{-s})(1-q^{1-s}),$ we see that
$$\lim_{i\to\infty}\frac{\log \zeta_{X_i}(s)}{2g_i}=\log \zeta_{\{X_i\}}(s)$$
for any $s$, such that $\Re s > \frac{1}{2}$ and $s\neq 1+ \frac{2\pi k}{\log q},$ $k\in \bbZ$ and the convergence is uniform in $a<|s-1|<b$ for small enough $a$ and $b.$ We use the Cauchy integral formula to get the statement of the corollary.
\end{proof}

\begin{remark}
It seems not completely uninteresting to study the behaviour of $\gamma_X^k$ ``on the finite level'', i.e. to try to obtain bounds on $\gamma_X^k$ for an individual curve $X.$ This was done in \cite{Ih} for $\gamma_X.$ In the general case the explicit version of the generalized Brauer--Siegel theorem from \cite{LZ} might be useful.
\end{remark}

\begin{remark}
It is worth noting that the above corollaries describe the relation between $\log \zeta_{X_i}(s)$ and $\log \zeta_{\{X_i\}}(s)$ near the point $s=1.$ The original statement of theorem \ref{BSGen} is stronger since it gives this relation for all $s$ with $\Re s > \frac{1}{2}.$
\end{remark}
\end{example}

\begin{example}[Varieties over finite fields]
Just as for curves, for varieties over finite fields we can get similar corollaries concerning the asymptotic behaviour of $\zeta_{X_i}(s)$ close to $s=d.$ We give just the statements, since the proofs are nearly the same as before.

The following result is the Brauer--Siegel theorem for varieties proven in \cite{Z1}.
\begin{corollary}
For an asymptotically exact family of varieties $\{X_i\}$ of dimension $n$ over a finite field $\bbF_q$ we have:
\begin{equation*}
\lim_{i\to\infty}\frac{\log |\varkappa_i|}{b(X_i)}= \sum_{f=1}^{\infty} \phi_f \log \frac{q^{fn}}{q^{fn}-1},
\end{equation*}
where $\varkappa_i=\Res\limits_{s=d}\zeta_{X_i}(s).$
\end{corollary}

In the next corollary we use the same definition of the Euler--Kronecker constants for varieties over finite fields as in the previous example for curves:

\begin{corollary}
For an asymptotically exact family of varieties $\{X_i\}$ of dimension $n$ we have  $\lim\limits_{i\to \infty}\frac{\gamma_{X_i}^k}{b(X_i)}= \gamma_{\{X_i\}}^k$
for any $k.$ In particular, $\lim\limits_{i\to \infty}\frac{\gamma_{X_i}}{b(X_i)}=-\sum\limits_{f=1}^{\infty} \frac{\phi_f f \log q}{q^{fn}-1}.$
\end{corollary}
\end{example}

\begin{example}[Elliptic curves over function fields]

Let us recall first the Brauer--Siegel type conjectures for elliptic curves over function fields due to Hindry--Pacheko \cite{HP} and Kunyavskii--Tsfasman \cite{KT}.

For an elliptic curve $E/K,$ $K=\bbF_q(X)$ we define $c_{E/K}$ and $r_{E/K}$ from $L_{E/K}(s)=c_{E/K}(s-1)^{r_{E/K}}+o\left((s-1)^{r_{E/K}}\right).$ The invariants $r_{E/K}$ and $c_{E/K}$ are important from the arithmetical point of view, since the geometric analogue of the Birch and Swinnerton-Dyer conjecture predicts that $r_{E/K}$ is equal to the rank of the group of $K$-rational points of $E/K$ and $c_{E/K}$ can be expressed via the order of the Shafarevich--Tate group, the covolume of the Mordell--Weil lattice (the regulator) and some other quantities related to $E/K$ which are easier to control.

\begin{conjecture}[Hindry--Pacheko]
\label{HinPa}
Let $E_i$ run through a family of pairwise non-isomorphic elliptic curves over a fixed function field $K.$ Then
$$\lim_{i\to\infty} \frac{\log |c_{E_i/K}|}{h(E_i)}=0,$$
where $h(E_i)$ is the logarithmic height of $E_i.$
\end{conjecture}

\begin{remark}
We could have divided $\log |c_{E_i/K}|$ by $n_{E_i}$ in the statement of the above conjecture since $h(E_i)$ and $n_{E_i}$ have essentially the same order of growth.
\end{remark}

\begin{conjecture}[Kunyavskii--Tsfasman]
\label{KunTsfBS}
For a family of elliptic curves $\{E_i/K_i\}$ obtained by a base change we have:
$$\lim_{i\to\infty}\frac{\log |c_{E_i/K_i}|}{g_{K_i}} = -\sum_{v\in X, f\geq 1}\phi_{v, f}\log \frac{\left|E_v(\bbF_{\rmN v^f})\right|}{\rmN v^f}.$$
\end{conjecture}

One can see that the above conjectures are both the statements of the type considered in the subsection \ref{centralpoint}. It is quite obvious for the first conjecture and for the second conjecture we have to use the explicit expression for the limit $L$-function:
$$\log L_{\{E_i/K_i\}}(s)=-\frac{1}{\nu+4}\sum_{v, f}\phi_{v, f} \log\left(1-(\alpha_v^f+\bar{\alpha}_v^f)\rmN v^{-f s}+\rmN v^{f(1-2s)}\right).$$

One can unify these two conjectures as follows:

\begin{conjecture}
\label{EllCurvesBS}
For an asymptotically very exact family of elliptic curves over function fields $\{E_i/K_i\}$ we have:
$$\lim_{i\to\infty}\frac{\log |c_{E_i/K_i}|}{d_i} = \log L_{\{E_i/K_i\}}(1),$$
where $d_i=n_{E_i}+4g_{K_i}-4$ is the degree of $L_{E_i/K_i}(s).$
\end{conjecture}

We are, however, sceptical about this conjecture holding for all families of elliptic curves. Theorems \ref{BSGen} and \ref{upperboundBS} imply the following result (a particular case of which was stated in \cite{Z2}) in the direction of the above conjectures:

\begin{theorem}
\label{KunTsfHin}
For an asymptotically very exact family of elliptic curves $\{E_i/K_i\}$ the following identity holds:
$$\lim_{i\to\infty}\frac{\log L_{E_i/K_i}(s)}{d_i} = \log L_{\{E_i/K_i\}}(s),$$
for $\Re s > 1.$ Moreover,
$$\lim_{i\to\infty}\frac{\log |c_{E_i/K_i}|}{g_i} \leq \log L_{\{E_i/K_i\}}(1).$$
\end{theorem}

\begin{remark}
If we consider split families of elliptic curves (i.e. $E_i=E\times X_i$, where $E/\bbF_q$ is a fixed elliptic curve) then the proof of theorem 2.1 from \cite{KT} gives us that the question about the behaviour of $L_{E_i/X_i}(s)$ at $s=1$ translates into the same question concerning the behaviour of $\zeta_{X_i}(s)$ on the critical line. For example, if the curve $E$ is supersingular, then conjecture \ref{EllCurvesBS} holds if and only if $\lim\limits_{i\to\infty}\frac{\log |\zeta_{X_i}(1/2)|}{g_i} = \log \zeta_{\{X_i\}}(1/2)$ (where $\zeta_{X_i}(\frac{1}{2})$ is understood as the first non-zero coefficient of the Taylor series expansion of $\zeta_{X_i}(s)$ at $s=\frac{1}{2}$). So, to prove the simplest case of conjecture \ref{EllCurvesBS} we have to understand the asymptotic behaviour of zeta functions of curves over finite fields on the critical line.
\end{remark}
\end{example}

\section{Distribution of zeroes}
\label{sectzero}

\subsection{Main results}

In this section we will prove certain results about the limit distribution of zeroes in families of $L$-functions. As a corollary we will see that the multiplicities of zeroes in asymptotically very exact families of $L$-functions can not grow too fast.

Let $C=C[-\pi, \pi]$ be the space of real continuous functions on $[-\pi, \pi]$ with topology of uniform convergence. The space of measures $\mu$ on $[-\pi, \pi]$ is by definition the space $\calM,$ which is topologically dual to $C.$ The topology on $\calM$ is the $*$-weak one: $\mu_i\to\mu$ if and only if $\mu_i(f)\to \mu(f)$ for any $f\in C.$

The space $C$ can be considered as a subspace of $\calM:$  if $\phi(x)\in C$ then $\mu_{\phi}(f)=\int_{-\pi}^{\pi} f(x)\phi(x)\,dx.$ The subspace $C$ is dense in $\calM$ in $*$-weak topology.

Let $L(s)$ be an $L$-function and let $\rho_1, \dots, \rho_d$ be the zeroes of the corresponding polynomial $\calL(u).$ Define $\theta_k\in(-\pi, \pi]$ by $\rho_k=q^{-w/2} e^{i \theta_k}.$ One can associate a measure to $L(s):$
\begin{equation}
\label{measdef}
\mu_{L}(f)=\frac{1}{d}\sum_{k=1}^{d}\delta_{\theta_k}(f),
\end{equation}
where $\delta_{\theta_k}$ is the Dirac measure supported at $\theta_k,$ i.e. $\delta_{\theta_k}(f)=f(\theta_k)$ for an $f\in C.$

The main result of this section is the following one:

\begin{theorem}
\label{zerodistrib}
Let $\{L_j(s)\}$ be an asymptotically very exact family of $L$-functions. Then the limit $\calM_{\bflim}=\lim\limits_{j\to\infty}\calM_j$ exists. Moreover, $\calM_{\bflim}$ is a nonnegative continuous function given by an absolutely and uniformly convergent series:
\begin{equation*}
\calM_{\bflim}(x)=1-2\sum_{k=1}^{\infty}\lambda_k \cos(kx) q^{-\frac{wk}{2}}.
\end{equation*}
\end{theorem}

\begin{proof}
The absolute and uniform convergence of the series follows from the definition of an asymptotically very exact family. It is sufficient to prove the convergence of measures on the space $C.$

The linear combinations of $\cos(m x)$ and $\sin(m x)$ are dense in the space of continuous functions $C,$ so it is enough to prove that for any $m=0, 1, 2, \dots$ we have:
\begin{equation}
\label{measeven}
\lim\limits_{j\to\infty}\calM_j(\cos(m x))=\calM_{\bflim}(\cos (m x)),
\end{equation}
and
\begin{equation}
\label{measodd}
\lim\limits_{j\to\infty}\calM_j(\sin(m x))=\calM_{\bflim}(\sin (m x)).
\end{equation}
The corollary \ref{corexplicit} shows that:
$$\calM_j(\cos(m x))=\sum_{k=1}^{d_j} \cos(m\theta_{k j}) = - 2 \Lambda_m q^{-\frac{w m}{2}}$$
for $m\neq 0$ and $\calM_j(1)=d_j.$ Dividing by $d_j$ and passing to the limit when $j\to \infty$ we get (\ref{measeven}).

Now, we note, that if $\rho=e^{i \theta},$ with $\theta \neq k\pi$ is a zero of $\calL(u)$ then $\rho=e^{i (\theta +\pi)}$ is also a zero of $\calL(u)$ with the same multiplicity. Thus $\calM_j(\sin(m x))=0=\calM_{\bflim}(\sin (m x))$ for any $j$ and $m$. So we get (\ref{measodd}) and the theorem is proven.
\end{proof}

\begin{corollary}
\label{rankbound}
Let $\{\zeta_j(s)\}$ be an asymptotically very exact family of zeta functions with $\epsilon_{w_{\bfe}}=1$ and let $r_j$ be the order of zero of $\zeta_j(s)$ at $s=\frac{w_{\bfe}}{2}.$ Then
$$\lim\limits_{j\to\infty} \frac{r_j}{\tilde{d}_j}=0.$$
\end{corollary}

\begin{proof}
Suppose that $\limsup \frac{r_j}{\tilde{d}_j}=\varepsilon>0.$ Taking a subsequence we can assume that $\lim\limits_{j\to\infty} \frac{r_j}{\tilde{d}_j}=\varepsilon.$ Taking a subsequence once again and using proposition \ref{zetaL} we can assume that we are working with an asymptotically very exact sequence of $L$-functions $\{L_j(s)\}=\{L_{w_{\bfe}j}(s)\}$ for which the same property concerning $r_j$ holds.

By the previous theorem $\lim\limits_{j\to\infty}\calM_j=\calM_{\bflim}.$ Let us take an even continuous non-negative function $f(x)\in C[-\pi, \pi]$ with the support contained in $(-\frac{\varepsilon}{\alpha}, \frac{\varepsilon}{\alpha}),$ where $\alpha=4\max\{\int_{-\pi}^{\pi}\calM_{\bflim}(x)\,dx, 1\}$ and such that $f(0)=1.$ We see that
$$\varepsilon \leq \lim_{j\to\infty}\calM_j(f(x))=\int_{-\pi}^{\pi}f(x)\calM_{\bflim}(x)\,dx \leq \frac{\varepsilon}{2},$$
so we get a contradiction. Thus the corollary is proven.
\end{proof}

\begin{remark}
It is easy to see that the same proof gives that the multiplicity of the zero at any particular point of the critical line grows asymptotically slower than $d.$
\end{remark}

\begin{remark}
\label{familyineq}
Using theorem \ref{zerodistrib} one can give another proof of the basic inequality for asymptotically very exact families of $L$-functions(corollary \ref{corineqL}). Indeed, all the measures defined by (\ref{measdef}) are non-negative. Thus the limit measure $\calM_{\bflim}$ must have a non-negative density at any point, in particular at $x=0.$ This gives us exactly the basic inequality. In this way we get an interpretation of the difference between the right hand side and the left hand side of the basic inequality as ``the asymptotic number of zeroes of $L_j(s),$ accumulating at $s=\frac{w}{2}$''.

In fact, using the same reasoning as before, we get a family of inequalities (which are interesting when not all the coefficients $\lambda_f$ are non-negative):
$$\sum_{k=1}^{\infty}\lambda_k \cos(kx) q^{-\frac{wk}{2}}\leq \frac{1}{2}$$
for any $x\in \bbR.$
\end{remark}

\begin{remark}
A thorough discussion of zero distribution results of similar type and their applications to various arithmetical problems can be found in \cite{Ser1}.
\end{remark}

\subsection{Examples}

\begin{example}[Curves over finite fields]
In the case of curves over finite fields we recover the theorem 2.1 from \cite{TV97}:

\begin{corollary}
For an asymptotically exact family $\{X_i\}$ of curves over a finite field $\bbF_q$ the limit $\calM_{\{X_i\}}=\lim\limits_{i\to\infty}\calM_{X_i}$ is a continuous function given by an absolutely and uniformly convergent series:
\begin{equation*}
\calM_{\{X_i\}}(x)=1-\sum_{k=1}^{\infty}k\phi_k h_k(x),
\end{equation*}
where
$$h_k(x)=\frac{q^{k/2}\cos(k x)-1}{q^k+1-2q^{k/2}\cos(k x)}.$$
\end{corollary}
\begin{proof}
This follows from theorem \ref{zerodistrib} together with the following series expansion:
$$\sum_{l=1}^{\infty}t^{-l}\cos(l k x)= \frac{t\cos(k x)-1}{t^2+1-2t\cos(k x)}.$$
\end{proof}
\end{example}

\begin{example}[Varieties over finite fields]
We can not say much in this case since the zero distribution theorem \ref{zerodistrib} applies only to $L$-functions. The only thing we get is that the multiplicity of zeroes on the line $\Re s = n-\frac{1}{2}$ divided by the sum of Betti numbers tends to zero (corollary \ref{rankbound}).
\end{example}

\begin{example}[Elliptic curves over function fields]

Let us consider first asymptotically bad families of elliptic curves. We have the following corollary of theorem \ref{zerodistrib}.

\begin{corollary}
\label{zeroesMi}
For an asymptotically bad family of elliptic curves $\{E_i\}$ over function fields the zeroes of $L_{E_i}(s)$ become uniformly distributed on the critical line when $i\to \infty.$
\end{corollary}
This result in the particular case of elliptic curves over the fixed field $\bbF_q(t)$ was obtained in \cite{Mi}. In fact, unlike us, Michel gives an estimate for the difference $\calM_i-\calM_{\{E_i\}}$ in terms of the conductor $n_{E_i}.$ It would be interesting to have such a bound in general.

\begin{corollary}
For an asymptotically very exact family of elliptic curves $\{E_i/K_i\}$ obtained by a base change the limit $\calM_{\{E_i/K_i\}}=\lim\limits_{i\to\infty}\calM_{E_i/K_i}$ is a continuous function given by an absolutely and uniformly convergent series:
\begin{equation*}
\calM_{\{E_i/K_i\}}(x)=1-\frac{2}{\nu+4}\,\sum_{v, f}\phi_{v, f} f d_v\sum_{k=1}^{\infty}\frac{\alpha_v^k+\bar{\alpha}_v^k}{q^{f d_v k}}\cos(f d_v k x).
\end{equation*}
\end{corollary}

\begin{corollary}
\label{rankboundell}
For a family of elliptic curves $\{E_i/K_i\}$ obtained by a base change
$$\lim\limits_{i\to\infty} \frac{r_i}{g_i}=0.$$
\end{corollary}
\begin{proof}
By proposition \ref{veryexactell} any such family contains an asymptotically very exact subfamily so we can apply corollary \ref{rankbound}.
\end{proof}

\begin{remark}
For a fixed field $K$ and elliptic curves over it a similar statement can be deduced from the bounds in \cite{Bru}. However, in the case of the base change Brumer's bounds do not imply corollary \ref{rankbound}. It would be interesting to see, what bounds one can get for the analytic ranks of individual elliptic curves when we vary the ground field $K$. Getting such a bound should be possible with a proper choice of a test function in the explicit formulae.
\end{remark}
\end{example}

\section{Open questions and further research directions}
\label{sectquestions}

In this section we would like to gather together the questions which naturally arise in the connection with the previous sections. Let us start with some general questions. First of all:

\begin{question}
To which extent the formal zeta and $L$-functions defined in section \ref{sectzeta} come from geometry?
\end{question}

One can make it precise in several ways. For example, it is possible to ask whether any $L$-function of weight $w,$ such that $\calL(u)$ has integral coefficients is indeed the characteristic polynomial of the Frobenius automorphism acting on the $w$-th cohomology group of some variety $V/\bbF_q.$ A partial answer to this question when $w=1$ is provided by the Honda--Tate theorem on abelian varieties \cite{Ta}.
\begin{question}
Describe the set $\{(\lambda_1, \lambda_2, \dots)\}$ for asymptotically exact (very exact) families of zeta functions ($L$-functions).
\end{question}

There are definitely some restrictions on this set, namely those given by various basic inequalities (theorems \ref{basicineqL} and \ref{basicineqzeta}, remark \ref{familyineq}). It would be interesting to see whether there are any others. We emphasize that the problem is not of arithmetic nature since we do not assume that the coefficients of polynomials, corresponding to $L$-functions, are integral. It would be interesting to see what additional restrictions the integrality condition on the coefficients of $\calL(u)$ might give. Note that, using geometric constructions, Tsfasman and Vl\u adu\c t \cite{TV97} proved that the sets of parameters $\lambda_f,$ satisfying $\lambda_f \geq 0$ for any $f$ and the basic inequality are all realized when $q$ is a square and $w=1$. This implies the same statement for $L$-functions with arbitrary $q$ and $w.$ However, our new $L$-function might no longer have integral coefficients.

\begin{question}
How many asymptotically good (very good) families are there among all asymptotically exact (very exact) families?
\end{question}

The ``how many'' part of the question should definitely be made more precise. One way to do this is to consider the set $V_g$ of the vectors of coefficients of polynomials corresponding to $L$-functions  of degree $d$ and its subset $V_d(f, a, b)$ consisting of the vectors of coefficients of polynomials corresponding to $L$-functions with $a<\frac{\Lambda_f}{d}<b.$ A natural question is whether the ratio of the volume of $V_d(f, a, b)$ to the volume of $V_g$ has a limit when $d\to\infty$ and what this limit is. See \cite{DH} for some information about $V_d.$ The question is partly justified by the fact that it is difficult to construct asymptotically good families of curves. We would definitely like to know why.

Let us now ask some questions concerning the concrete results on zeta and $L$-functions proven in the previous sections.

\begin{question}
Is it true that the generalized Brauer--Siegel theorem \ref{BSGen} holds on the critical line for some (most) asymptotically very exact families?
\end{question}

It is sure that without the additional arithmetic conditions on the family the statement does not hold. The most interesting families here are the families of elliptic curves over function fields considered in subsection \ref{exBS} due to the arithmetic applications. An example of a family of elliptic surfaces for which the statement holds is given in \cite{HP}. It is interesting to look at some other particular examples of families of curves over finite fields where the corresponding zeta functions are more or less explicitly known. These include the Fermat curves \cite{KS} and the Jacobi curves \cite{Ko}.

Some examples we know to support the positive answer to the above question come from the number field case. It is known that there exists a sequence $\{d_i\}$ in $\bbN$ of density at least $\frac{1}{3}$  such that
$$\lim\limits_{i\to\infty}\frac{\log \zeta_{\bbQ(\sqrt{d_i})}(\frac{1}{2})}{\log d_i}=0$$
(cf. \cite{IS}). The techniques of the evaluation of mollified moments of Dirichlet $L$-functions used in that paper is rather involved. It would be interesting to know whether one can obtain analogous results in the function field case. The related questions in the function field case are studied in \cite{KS}. It is not clear whether the results on the one level densities for zeroes obtained there can be applied to the question of finding a lower bound on $\frac{\log |c_i|}{d_i}$ for some positive proportion of fields (both in the number field and in the function field cases).

\begin{question}
Prove the generalized Brauer--Siegel theorem \ref{BSGen} with an explicit error term.
\end{question}

This was done for curves over finite fields in \cite{LZ} and looks quite feasible in general. It is also worth looking at particular applications that such a result might have, in particular one may ask what bounds on the Euler--Kronecker constants it gives.

\begin{question}
How to characterize measures corresponding to asymptotically very exact families?
\end{question}

This was done in \cite{TV97} for families such that $\lambda_f\geq 0$ for all $f.$ The general case remains open.

\begin{question}
Estimate the error term in theorem \ref{zerodistrib}.
\end{question}

As it was mentioned before, in the case of elliptic curves over $\bbF_q(t)$ the estimates were carried out in \cite{Mi}.

\begin{question}
Find explicit bounds on the orders of zeroes of $L$-functions on the line $\Re s=\frac{w}{2}.$
\end{question}

The corollary \ref{rankbound} gives that the ratio $\frac{r_i}{d_i}\to 0$ for asymptotically very exact families (here $r_i$ is the multiplicity of the zero). In a particular case of elliptic curves over a fixed function field Brumer in \cite{Bru} gives a bound which grows asymptotically slower than the conductor. Using explicit formulae with a proper choice of test functions, it should be possible to give such upper bounds for families obtained by a base change if not in general.

Let us finally ask a few more general questions.

\begin{question}
How can one apply the results of this paper to get the information about the arithmetic or geometric properties of the objects to which $L$-functions are associated?
\end{question}
We carried out this task (to a certain extent) in the case of curves and varieties over finite fields and elliptic curves over function fields. Additional examples are more than welcome.

The last but not least:

\begin{question}
What are the number field analogues of the results obtained in this paper?
\end{question}

It seems that most of the results can be generalized to the framework of the Selberg class (as described, for example, in \cite[Chapter 5]{IK}), subject to imposing some additional hypothesis (such as the Generalized Riemann Hypothesis, the Generalized Ramanujan Conjectures, etc.). Of course, one will have to overcome quite a lot of analytical difficulties on the way (compare, for example, \cite{TV97} and \cite{TV02}).

We hope to return to this interesting and promising subject later on.

\subsection*{Acknowledgements.} I would like to thank my teachers Michael Tsfasman and Serge Vl\u adu\c t who taught me the asymptotic theory of global fields. The discussions with them and their constant advises were of great value.

\end{document}